\documentclass[10pt, twoside]{amsart}

\usepackage[utf8]{inputenc}
\usepackage{hyperref}

\usepackage{graphicx}              
\usepackage{amsmath, bm}               
\usepackage{amsfonts}              
\usepackage{amsthm}                
\usepackage{amssymb}
\usepackage{mathtools}             

\usepackage{enumitem} 

\usepackage[all]{xy}
\usepackage{amscd}

\allowdisplaybreaks
\usepackage{color}
\usepackage{fp}

\newcommand{\Z}{\mathbb{Z}}

\newcommand{\C}{\mathbb{C}}

\newcommand{\Hi}{\mathcal{H}}
\newcommand{\Li}{\mathcal{L}}
\newcommand{\X}{\mathsf{X}}
\newcommand{\Y}{\mathsf{Y}}

\newcommand{\cj}{\overline}
\newcommand{\op}{\mathrm}

\DeclarePairedDelimiterX{\normb}[1]{\lVert}{\rVert}{#1}

\numberwithin{equation}{section}
\newtheorem{theorem}{Theorem}[section]
\newtheorem{lemma}[theorem]{Lemma}
\newtheorem{proposition}[theorem]{Proposition}
\newtheorem{corollary}[theorem]{Corollary}

\let\origproofname\proofname
\renewcommand{\proofname}{\upshape\textbf{\origproofname}}

\theoremstyle{definition}
\newtheorem{example}[theorem]{Example}
\newtheorem{definition}[theorem]{Definition}
\newtheorem{remark}[theorem]{Remark}
\newtheorem{nota}[theorem]{Notation}

\usepackage{indentfirst}

\begin{document}

\title{Representations of C*-correspondences on pairs of Hilbert spaces.}

\author{Alonso Delfín}

\date{\today}

\address{Department of Mathematics, University  of Oregon,
       Eugene OR 97403-1222, USA\\
      Department of Mathematics, University of Colorado, Boulder CO 80309-0395, USA.}

\email[]{alonsod@uoregon.edu \\
alonso.delfin@colorado.edu}

\thanks{\textsc{Department of Mathematics, University of Oregon, Eugene OR 97403-1222, USA\\
Department of Mathematics, University of Colorado, Boulder CO 80309-0395, USA}
}

\begin{abstract} 
We study representations of Hilbert bimodules on pairs of Hilbert spaces. 
If $A$ is a C*-algebra and $\mathsf{X}$ is a right Hilbert $A$-module,
we use such representations to 
faithfully represent the C*-algebras 
$\mathcal{K}_A(\mathsf{X})$ and $\mathcal{L}_A(\mathsf{X})$. 
We then extend this theory to define representations 
of $(A,B)$ C*-correspondences on a pair of Hilbert spaces 
and show how these can be obtained from any nondegenerate
representation of $B$.
As an application of such representations, 
we give necessary and sufficient conditions on 
an $(A,B)$ C*-correspondences to admit a Hilbert $A$-$B$-bimodule structure. 
Finally, we show how to represent
the interior tensor product of two C*-correspondences. 
\end{abstract}
\maketitle

 \section{Introduction}

Let $A$ and $B$ be C*-algebras. Research on normed $A$-$B$-bimodules usually 
deals with either Hilbert $A$-$B$-bimodules or $(A,B)$ C*-correspondences.
In this paper we study representations of both kinds of bimodules on pairs of Hilbert spaces, 
which roughly consists of realizing the bimodule as a closed subspace of 
$\Li(\Hi_0, \Hi_1)$, 
the space of bounded linear maps from a Hilbert space $\Hi_0$ 
to a Hilbert space $\Hi_1$. 
The Hilbert bimodule case was introduced by 
R. Exel \cite{Exel1993}. We generalize Exel's definition 
to the general C*-correspondence setting. 
Representations of $(A,A)$ C*-correspondences on 
a Hilbert space, 
sometimes also called covariant representations, have been studied thoroughly in the past. See for instance 
\cite{pim1997}, \cite{MuhlySolel1998}, and \cite{kat2004}. 
These representations are the main object used to construct the
Cuntz-Pimsner algebra of an $(A,A)$ C*-correspondence. 
Our definition for representations of $(A,B)$ C*-correspondences on a pair of Hilbert spaces
agrees with the other definitions in the literature when $A=B$. 
We use representations of $(A,B)$ C*-correspondences on pairs of Hilbert 
spaces as a tool for mainly three purposes:
\begin{enumerate}
\item Give representations of adjointable and compact-module maps on right Hilbert modules. 
\item Answer when a general $(A,B)$ C*-correspondence can be uniquely given the structure of a Hilbert $A$-$B$-bimodule.
\item Represent the interior tensor product of C*-correspondences as the product of suitable representations of the factors.
\end{enumerate} 

Unfortunately, in the early literature, 
$(A,A)$ C*-correspondences were sometimes
also referred to as Hilbert bimodules over $A$. 
See for instance \cite{pim1997} and \cite{KajPinWat1998}.
In this paper we follow the current naming conventions. 
That is, for us 
both Hilbert $A$-$B$-bimodules and $(A,B)$ C*-correspondences come equipped with a 
$B$-valued right inner product
which in turn defines a norm on the module.
The difference between them lies in the fact that Hilbert $A$-$B$-bimodules  
come with an $A$-valued left inner product compatible with the $B$-valued one, 
while for $(A,B)$ C*-correspondences 
we require $A$ to act as adjointable operators, but an $A$-valued inner product 
is not assumed.

By a representation of a Hilbert $A$-$B$-bimodule $\X$ on 
a pair of Hilbert spaces $(\Hi_0, \Hi_1)$, we mean a triple of maps $(\lambda_A, \rho_B, \pi_\X)$, where
$\lambda_A$ is a representation of $A$ on $\Hi_1$,
$\rho_B$ is a representation of $B$ on $\Hi_0$, and
$\pi_\X: \X \to \Li(\Hi_0, \Hi_1)$ is a linear map,
such that $\pi_\X(\X)$ has the Hilbert $(\lambda_A(A), \rho_B(B))$-bimodule
structure where both module actions and both inner
products are given by multiplication of operators. 
See Definition \ref{RepBimodDef} for more details. 
That such representations do exist is shown in
Propositions 4.7 and 4.8 of \cite{Exel1993}.
Since every Hilbert $A$-$B$-bimodule is 
in particular an $(A,B)$ C*-correspondence, a natural question arises here: can we 
also represent an $(A,B)$ C*-correspondence on a pair of Hilbert spaces
in a way that generalizes representations of Hilbert bimodules? We give 
a positive answer to this question 
in Theorem \ref{CorrespRep} by establishing 
that any $(A,B)$ C*-correspondence can be represented on a particular
pair of Hilbert spaces $(\Hi_0, \Hi_1)$. 
In the $(A,B)$ C*-correspondence cases, as in Exel's result for Hilbert bimodules, 
the Hilbert space $\Hi_0$ 
will also come from a given nondegenerate representation of $B$ 
and $\Hi_1$ will be obtained so that $A$ is nondegenerately represented on $\Hi_1$.
The methods we use differ significantly 
from those used by Exel in \cite{Exel1993}, Murphy in \cite{murphy_1997}, and Zettl in \cite{Zettl1983}
for their analogous results for Hilbert bimodules and Hilbert modules. 
In particular, our methods do not rely on the linking algebra of a Hilbert bimodule $\X$, 
nor on the theory of positive definite kernels for Hilbert modules,
nor on the dual module $\X''$ of a right Hilbert $A$-module $\X$, 
which is a right Hilbert $A^{**}$-module.
Nevertheless, our methods
can easily be adapted to work in the Hilbert bimodule setting, 
making our notion of representations of C*-correspondences more general. 
Indeed, in Theorem \ref{BimodRepInCorrespRep} we adapt our methods from 
the C*-correspondence case to show the existence of a representation $(\lambda_A, \rho_B, \pi_\X)$ for any Hilbert $A$-$B$-bimodule $\X$
on a pair of Hilbert spaces $(\Hi_0, \Hi_1)$. 
In contrast with the analogue result from \cite{Exel1993}, we give an explicit formula for the map $\pi_\X$. All this is used in our proofs for 
Theorem \ref{Suff_Nec_Cond_Corres_is_Bimod} and Corollary \ref{COR_Suff_Nec_Cond_Corres_is_Bimod}, where we give necessary and sufficient conditions for an
$(A,B)$ C*-correspondence to uniquely admit a Hilbert $A$-$B$-bimodule structure.  

An important tool when working with C*-correspondences is their interior tensor product. 
We show, in Theorem \ref{TensorCorres}, that representations of  C*-correspondences
are well behaved with respect to the interior tensor product. By this we mean that 
it is always possible to find a representation of the tensor product correspondence by looking 
at suitable representations of the factors. Moreover,
if $(\X, \varphi_\X)$ and $(\Y, \varphi_\Y)$ 
are C*-correspondences whose interior tensor product makes sense, we can always 
think of $\X \otimes_{\varphi_\Y} \Y$ as $\cj{\tau_\X(\X) \pi_\Y(\Y)}$, 
where $\tau_\X$ and $\pi_\Y$ come from carefully chosen representations of 
$(\X, \varphi_\X)$ and $(\Y, \varphi_\Y)$. In practice, 
this makes the tensor product more manageable, as elementary tensors are now 
replaced by compositions of operators. This can be compared with Theorem 3.2 in 
\cite{murphy_1997}, where Murphy uses that right Hilbert modules can 
be concretely represented on pairs of Hilbert spaces to give an elementary 
construction of the exterior tensor product of right Hilbert modules. 

A main advantage of having a right Hilbert module represented as a subspace 
of $\Li(\Hi_0, \Hi_1)$ is that, assuming 
some nondegeneracy conditions, the C*-algebras of adjointable maps 
and compact-module maps of the module can be faithfully represented on $\Hi_1$. 
Indeed, this is shown in Propositions \ref{K_AX}, \ref{L_AX}, and \ref{RepforKandL}.
These representations have not been studied in the current literature. However, 
they play an important role in current work dealing with modules and 
correspondences over $L^p$-operator algebras (See Chapters V and VI in \cite{Del2023}). 

In fact, our main motivation to study representations of C*-correspondences
on a pair of Hilbert spaces is that they naturally give a potential definition of 
what a correspondence might be when replacing C*-algebras by 
general Banach algebras. For instance, in Definition \ref{RepMod} one can
replace Hilbert spaces by $L^p$-spaces
and define a module over an $L^p$-operator algebra to be a pair of subspaces 
that satisfy the conditions satisfied by the pair $(\pi_\X(\X), \pi_\X(\X)^*)$ in Definition \ref{RepMod}. 
Motivated by the results in this paper, the author's doctoral dissertation, see \cite{Del2023},
considers a version of the
Cuntz-Pimsner algebras for $L^p$-operator algebras. 
Indeed, for $p \in (1, \infty)$ and $d \in \Z_{\geq 2}$, it is possible to use a 
construction similar to the Cuntz-Pimsner construction, see \cite{pim1997} and \cite{kat2004},
but on a $\Li(\ell^p_1)$-module to get $\mathcal{O}_d^p$, the $L^p$-version of the Cuntz algebra introduced in \cite{ncp2012AC}.  Many results from this paper 
have motivated definitions for our investigations of modules over $L^p$-operator algebras, $L^p$-correspondences 
and the $L^p$-operator algebras generated by these. For details see Chapters V and VI in \cite{Del2023}.

\textbf{Structure of the paper:} In Section \ref{s1} we start with the particular case of right Hilbert modules. 
Roughly speaking, for a pair of Hilbert spaces $(\Hi_0,\Hi_1)$ and 
a concrete C*-algebra $A \subseteq \Li(\Hi_0)$, 
we analyze the behavior of closed subspaces of $\mathcal{L}(\Hi_0, \Hi_1)$
that have the obvious structure of right Hilbert $A$-modules. We pay close 
attention to the adjointable maps and compact-module maps 
of these closed subspaces. 
We point out in Section \ref{s2} that any right Hilbert $A$-module 
can be represented as an isometric copy of a closed subspace of 
$\mathcal{L}(\Hi_0, \Hi_1)$ for a pair of Hilbert spaces $(\Hi_0,\Hi_1)$. 
In fact, we see that this is true for any Hilbert bimodule. 

In Section \ref{s3} we introduce representations of 
$(A,B)$ C*-correspondences on a pair of Hilbert spaces. 
We prove that this is indeed a generalization of 
Exel's theory of representations for 
Hilbert $A$-$B$-bimodules. This yields 
a framework for representations of
general C*-correspondences
which, as in the Hilbert bimodule case, shows that any C*-correspondence
is isometrically isomorphic to a closed subspace of 
$\mathcal{L}(\Hi_0, \Hi_1)$ for a pair of Hilbert spaces $(\Hi_0,\Hi_1)$. 
As an application, we use representations of C*-correspondences to show that 
an $(A,B)$ C*-correspondence $(\X, \varphi)$ admits 
a unique structure of a Hilbert $A/\op{ker}(\varphi)$-$B$-bimodule if and only if $\mathcal{K}_B(\X) \subseteq \varphi(A)$. This in turn gives 
necessary and sufficient conditions for $(\X, \varphi)$ to admit a unique structure of a Hilbert $A$-$B$-bimodule. 
We end the paper by showing that, given any two C*-correspondences that
share the middle action, it is always possible to represent their interior tensor 
product on a pair of Hilbert spaces by using particular representations 
of the original C*-correspondences. 

We end our introduction by establishing some of our notational conventions. 

Let $a \colon V_0 \to V_1$ be a linear map between vector spaces. In this case we follow the common convention of suppressing parentheses for linear maps and write $a \xi$ for the action of $a$ on $\xi \in V_0$. However, if $\X$ and $\Y$ are vector spaces that are also modules over an algebra $A$ and $t \colon  \X \to \Y$ is a linear module map, then we write $t(x)$ for the action of $t$ on $x \in \X$. This is needed to avoid some potential confusion when both $x$ and $t(x)$ happen to also be linear maps between vector spaces, which will occur frequently in this paper.  

If $\X$ is a subspace of linear maps between vector spaces $V_0$ and $V_1$, the product $\X V_0$ is defined as 
the linear span of elements in $\X$ acting on vectors from $V_0$, that is 
\[
\X V_0 = \op{span}\{ x\xi \colon x \in \X \mbox{ and } \xi \in V_0\} \subseteq V_1.
\] 

\begin{nota}\label{HMconv}
We fix some terminology for Hilbert modules over C*-algebras. 
The $A$-valued right inner product for a right Hilbert module 
will be denoted by $\langle - , - \rangle_A$. The map $(x,y) \mapsto \langle x , y \rangle_A$ is assumed 
to be linear in the second variable and conjugate linear 
in the first one. Similarly, the $A$-valued left 
inner product for a left Hilbert module will be denoted by ${}_A\langle - , - \rangle$. 
The map $(x,y) \mapsto {}_A\langle x , y \rangle$
is assumed 
linear in the first variable and conjugate linear 
in the second one. If $\X$ is any right Hilbert $A$-module, we use $\mathcal{L}_A(\X)$
to denote adjointable maps from $\X$ to itself. For each $x, y \in \X$ we have the
``rank one'' operator $\theta_{x,y}\in \Li_A(\X)$,
given by $\theta_{x,y}(z)=x\langle y, z\rangle_A$ for any $z \in \X$.  
We write $\mathcal{K}_A(\X)$ for the compact-module maps from $\X$ to itself,  
which are defined as the closed linear span of the ``finite rank'' operators. That is, 
\[
\mathcal{K}_A(\X) = \cj{\op{span}\{ \theta_{x,y} \colon x,y \in \X\}}.
\]
\end{nota}

Finally, we regard Hilbert spaces as right Hilbert $\C$-modules. 
For this reason, our convention for inner products of Hilbert spaces is the physicist's one: they are linear in the second variable and conjugate linear in 
the first one. 

\section{Concrete Hilbert Modules}\label{s1}

In this section, we describe a concrete example of a right Hilbert module $\X$
over a concrete C*-algebra $A \subseteq \Li(\Hi_0)$ for a Hilbert space $\Hi_0$. 
This is Example \ref{mainE} below, which was studied by 
G. J. Murphy in
Section 3 of \cite{murphy_1997}, where it was referred to as a 
\textit{Concrete Hilbert C*-Module}. 
We then provide useful representations for the C*-algebras 
$\Li_A(\X)$ and $\mathcal{K}_A(\X)$.
In Section \ref{s2} we will see that any right Hilbert $A$-module can be represented in this fashion.

\begin{example}\label{mainE}
Let $\Hi_0, \Hi_1$ be Hilbert spaces and let $A \subseteq \Li(\Hi_0)$ be a concrete 
C*-algebra. Suppose that $\X \subseteq \Li(\Hi_0, \Hi_1)$ is a closed subspace such
that $xa \in \X$ for all $x \in \X$ and all  $a\in A$, and such that $x^*y \in A$ for all
$x,y \in \X$.  For each $x, y \in \X$ we put 
\begin{equation}\label{IP_mainE}
\langle x, y \rangle_A = x^*y \in A.
\end{equation}
\end{example}
\begin{proposition}\label{mainEisHM}
Let $\X$ be as in Example \ref{mainE}. Then, $\X$ is a right Hilbert $A$-module
with $A$-valued inner product given by equation (\ref{IP_mainE}) and 
$\|\langle x, x \rangle_A \|^{1/2} = \|x\|$. 
\end{proposition}
\begin{proof}
It is clear that $\X$ is a right $A$-module. It is easily checked that $(x,y) \mapsto \langle x,y\rangle_A$ satisfies all the axioms of an 
$A$-valued inner product on $\X$. We claim that $\X$ is complete with the induced norm 
$\|x\|_A=\|\langle x, x \rangle_A \|^{1/2}$. Indeed, elements of the 
C*-algebra $\Li(\Hi_0 \oplus \Hi_1)$ can be written as $2 \times 2$ operator 
valued matrices and 
$\Li(\Hi_0, \Hi_1)$ is 
isometrically isomorphic to the lower left corner of $\Li(\Hi_0 \oplus \Hi_1)$, while $\Li(\Hi_0)$ is isomorphic to the upper left corner.
Hence, if $x \in \X$, the C*-equation at the second step yields 
\[
\| x\|^2 = \Bigl\| \begin{pmatrix}
0 & 0 \\
x & 0
\end{pmatrix}\Bigr\|^2 
=
 \Bigl\| \begin{pmatrix}
x^*x & 0 \\
0 & 0
\end{pmatrix}\Bigr\|  
 = 
\| x^*x \| = \|x\|^2_A.
\]
The claim now follows because $\X$ is closed in $\Li(\Hi_0, \Hi_1)$.
Thus, $\X$ is indeed a right Hilbert $A$-module. 
\end{proof}

\begin{remark}
Thanks to Proposition \ref{mainEisHM} above, when 
$\X$ is as in Example \ref{mainE}, we are free to not make any distinction between the norm $x \in X$
has as an element of $\Li(\Hi_0, \Hi_1)$ and the module norm $\|x\|_A$. Thus, 
from now on we drop the subscript $A$ and simply write $\|x\|$. 
\end{remark}

Next, we will show that the compact-module maps and adjointable maps 
of the Hilbert module in Example \ref{mainE} above can be realized as closed C*-subalgebras of $\Li(\Hi_1)$, 
provided that some nondegeneracy conditions hold. 

\begin{proposition}\label{K_AX}
Let $\X$ be the right Hilbert $A$-module described in Example \ref{mainE} above. 
Suppose that $\X\Hi_0$
is dense in $\Hi_1$. Then, there is a $*$-isomorphism from $\mathcal{K}_A(\X)$ to
\[
 \cj{\op{span} \left\{ xy^* \colon x,y \in \X \right\}} \subseteq \Li(\Hi_1)
\] 
which sends $\theta_{x,y}$ to $xy^*$ for $x,y \in \X$.
\end{proposition}
\begin{proof}
Define subspaces $K_1 \subseteq \Li(\Hi_1)$ 
and  $K_2\subseteq \mathcal{K}_A(\X)$ by letting 
\[
K_1=\op{span} \left\{ xy^* \colon x,y \in \X \right\}, K_2=\op{span} \left\{ \theta_{x,y} \colon x,y \in \X \right\}.
\]
Recall that $K_2$ is dense in $\mathcal{K}_A(\X)$.
Let $n \in \Z_{\geq 1}$ and let $x_1, \ldots, x_n, y_1, \ldots, y_n \in \X$.
Then, for any $z \in \X$ 
\[
\Bigl\| \sum_{j=1}^n\theta_{x_j,y_j}(z) \Bigr\|_{\X}=
\Bigl\| \Bigl(\sum_{j=1}^n x_jy_j^*\Bigr)z\Bigr\|_{\Li(\Hi_0,\Hi_1)} 
\leq \Bigl\|\sum_{j=1}^n x_jy_j^*\Bigr\|_{\Li(\Hi_1)}\|z\|.
\] 
This implies
\begin{equation}\label{WELLDEF1}
\Bigl\| \sum_{j=1}^n\theta_{x_j,y_j} \Bigr\|_{\mathcal{K}_A(\X)} \leq 
\Bigl\|\sum_{j=1}^n x_jy_j^* \Bigr\|_{\Li(\Hi_1)}.
\end{equation}
Let $\iota\colon K_1 \to \mathcal{K}_A(\X)$ be the linear extension of the map 
which sends  $xy^*$ to $\theta_{x,y}$ for $x,y \in \X$.
That is,
 \[
 \iota\Bigl(\sum_{j=1}^nx_jy_j^*\Bigr)=\sum_{j=1}^n\theta_{x_j,y_j}.
 \] 
That $\iota$ is well defined follows from (\ref{WELLDEF1}). 
In fact, (\ref{WELLDEF1}) gives $\| \iota(k) \| \leq \|k\|$ for all $k \in K_1$.
Thus, we can extend $\iota$ by continuity to a map $\tilde{\iota} \colon \cj{K_1} \to \mathcal{K}_A(\X)$
such that $\|\tilde{\iota}(s)\|_{\mathcal{K}_A(\X)} \leq \| s \|_{\Li(\Hi_1)}$
for all $s \in \cj{K_1}$. Our goal 
is to show that $\tilde{\iota}$ is a $*$-isomorphism from $\cj{K_1}$ to $\mathcal{K}_A(\X)$.
Notice that $\tilde{\iota}$ is already a $*$-homomorphism between C*-algebras. 
We will show that $\tilde{\iota}$ is injective, which in turn will make $\tilde{\iota}$ an isometry. 
Since $\tilde{\iota}$ maps $K_1$ onto $K_2$, a dense subset of $\mathcal{K}_A(\X)$, 
proving injectivity will automatically show that $\tilde{\iota}$ is a $*$-isomorphism and this 
will finish the proof. 

Take any $s \in \cj{K_1}$ and fix $x \in \X$. We claim that
the element $sx \in \Li(\Hi_0, \Hi_1)$ is actually in $\X$ and that it is equal to $\tilde{\iota}(s)(x) \in \X$. 
Indeed, for any $k \in K_1$, the element  $kx \in \Li(\Hi_0, \Hi_1)$ is an element of $\X$ 
(because $x_1x_2^*x \in \X A \subseteq \X$ for all $x_1,x_2 \in \X$) and it coincides with $\iota(k)(x) \in \X$.
Thus, by continuity it follows that $sx=\tilde{\iota}(s)(x)$, as claimed.

Finally, to prove that $\tilde{\iota}$ is injective, 
let $s \in \cj{K_1}$ satisfy $\tilde{\iota}(s)=0$ in $\mathcal{K}_A(\X)$. 
We have to show that $s=0$ in $\Li(\Hi_1)$, but since $\X \Hi_0$ is dense in $\Hi_1$, 
it is enough to prove that $s(x\xi)=0$ for all $x \in \X$ and $\xi \in \Hi_0$. 
Indeed, thanks to our last claim, we have 
\[
s(x\xi) = sx(\xi)=[\tilde{\iota}(s)(x)]\xi=0.
\]
This finishes the proof. 
\end{proof}

Before characterizing $\Li_A(\X)$, we need to recall a useful lemma 
and prove a general result about the direct sum of Hilbert modules. 

\begin{lemma}\label{tx_tx}
Let $A$ be a C*-algebra and let $\X$ be any Hilbert $A$-module. 
Then, for any $t \in \Li_A(\X)$ and any $x \in \X$, 
we have $\langle t(x), t(x)\rangle_A \leq \|t\|^2 \langle x,x \rangle_A$.
\end{lemma}
\begin{proof}
See Proposition 1.2 in \cite{lance_1995}.
\end{proof}

Let $A$ be a C*-algebra, let $\X$ be any right Hilbert 
$A$-module, and let $n \in \Z_{\geq 1}$. 
The direct sum $\X^n$ is usually regarded as a 
right Hilbert $A$-module in an obvious way. 
However, $\X^n$ can also be identified with 
$M_{1,n}(\X)$,
 the row vectors with $n$ entries in $\X$. 
 This identification makes $\X^n$ a right Hilbert $M_n(A)$-module, with
action that comes from the formal matrix multiplication 
$M_{1, n}(\X) \times M_n(A) \to M_{1,n}(\X)$. That is, 
\[
(x_1, \ldots, x_n)\cdot (a_{i,j})_{i,j} = 
\Bigl( \sum_{i=1}^n x_ia_{i,1}, \ldots,\sum_{i=1}^n x_ia_{i,n} \Bigr).
\]
The $M_n(A)$-valued right inner product comes
from the formal matrix multiplication 
$M_{n,1}(\X) \times M_{1, n}(\X) \to M_{n}(A)$. That is, 
\[
\langle (x_1, \ldots, x_n), (y_1, \ldots, y_n)\rangle_{M_n(A)} 
= (\langle x_i, y_j \rangle_A)_{i,j}.
\]
The following result should be well known. We include a complete proof as we 
couldn't find one in the current literature. 
\begin{proposition}\label{M_nAmod}
Let $A$ be a C*-algebra, let $\X$ be any right Hilbert 
$A$-module, and let $n \in \Z_{\geq 1}$. 
For each $t \in \Li_A(\X)$, we define a map $\kappa(t) \colon \X^n \to \X^n$ by
\[
\kappa(t)(x_1, \ldots, x_n)=\bigl(t(x_1), \ldots, t(x_n)\bigr).
\]
Then, $\kappa(t) \in \Li_{M_n(A)}(\X^n)$, and the map $t \mapsto \kappa(t)$ from $\Li_A(\X)$ to $\Li_{M_n(A)}(\X^n)$ 
is a $*$-isomorphism. 
\end{proposition}
\begin{proof}
Firstly we show that $\kappa(t) \in \Li_{M_n(A)}(\X^n)$. 
Indeed, an immediate calculation shows that
\[
\langle \kappa(t)(x_1, \ldots, x_n), (y_1, \ldots, y_n) \rangle_{M_n(A)} =
\langle (x_1, \ldots, x_n), \kappa(t^*)(y_1, \ldots, y_n) \rangle_{M_n(A)}.
\]
Therefore, $\kappa(t) \in \Li_{M_n(A)}(\X^n)$ and $\kappa(t)^*=\kappa(t^*)$.
It is now easily checked that $\kappa$ is in fact an injective $*$-homomorphism. 
Thus, to be done, we only need to show that $\kappa$ is surjective. 
We establish some notation first. 
For any $x \in \X$ and any $j \in \{ 1, \ldots, n\}$, we denote by
$\delta_jx$ the element of $\X^n$ with $x$ in the $j$-th coordinate 
and zero elsewhere. Thus, 
\[
(x_1, \ldots, x_n)=\sum_{j=1}^n \delta_jx_j.
\]
Now, take any $s \in \Li_{M_n(A)}(\X^n)$. We have 
linear maps $s_1, \ldots, s_n \colon \X^n \to \X$ such that 
\[
s(x_1, \ldots, x_n)=\bigl( s_1(x_1, \ldots, x_n), \ldots, s_n(x_1, \ldots, x_n)\bigr).
\]
For each $i,j \in \{1, \ldots, n\}$, we define a linear map $s_{i,j} \colon \X \to \X$
by letting $s_{i,j}(x)=s_i(\delta_j x)$ for any $x \in \X$. 
Therefore, 
\[
s(x_1, \ldots, x_n) = \sum_{j=1}^n \bigl( s_{1,j}(x_j), \ldots, s_{n,j}(x_j)\bigr).
\]
Notice that since $s$ is adjointable, we have a map 
$s^* \colon \X^n \to \X^n$, which in turn gives,
for each $i,j \in \{1, \ldots, n\}$, a linear map 
$(s^*)_{i,j} \colon \X \to \X$. 
The equation 
\[
\langle s(x_1, \ldots, x_n) , (y_1, \ldots, y_n) \rangle_{M_{n}(A)} = \langle (x_1, \ldots, x_n) , s^*(y_1, \ldots, y_n) \rangle_{M_{n}(A)},
\]
becomes
\begin{equation}\label{ss_star}
\left(\Bigl\langle \sum_{k=1}^n s_{i,k}(x_k), y_j \Bigr\rangle_A \right)_{i,j}=
\left(\Bigl\langle x_i, \sum_{k=1}^n (s^*)_{j,k}(y_k)\Bigr\rangle_A\right)_{i,j}.
\end{equation}
In particular,  let $l,m \in \{1, \ldots, n\}$
be such that $l \neq m$. Take any $x,y \in \X$ and notice that the $(l,l)$ entry
in (\ref{ss_star}), applied to the elements $\delta_m x \in \X^n$ and $\delta_l y \in \X^n$, 
becomes the equation
\[
\langle s_{l,m}(x), y \rangle_A= \langle 0, (s^*)_{l,l}(y) \rangle_A = 0.
\]
Thus, for each  $l,m \in \{1, \ldots, n\}$
with $l \neq m$, 
we have shown that $s_{l,m}=0$. An analogous 
computation also shows that $(s^*)_{l,m}=0$ when $l \neq m$. Then, (\ref{ss_star}) implies
\[
\langle s_{i,i}(x_i), y_j \rangle_A = \langle x_i , (s^*)_{j,j}(y_j) \rangle_A
\]
for all $i,j \in \{1, \ldots, n\}$. It now follows at once that, for all $i \in \{1, \ldots, n\}$, 
$s_{i,i} \in \Li_A(\X)$ with $(s_{i,i})^*=(s^*)_{i,i}$. Furthermore, 
this also proves that $s_{i,i}=s_{j,j}$ for all $i,j \in \{1, \ldots, n\}$.
It is now clear that $\kappa(s_{i,i})=s$ for any $i \in \{1, \ldots, n\}$, 
which finishes the proof. 
\end{proof}

\begin{remark}\label{Daws_Comment3}
On page 39 of \cite{lance_1995} it is claimed, with no proof, 
that $\Li_A(\X^n) \cong \Li_{M_n(A)}(\X^n)$. 
Proposition \ref{M_nAmod} shows that the claim is false 
in general. Indeed, if $A=\X=\C$ and $n \in \Z_{\geq 2}$, 
then it is clear that $\Li_\C(\C^n) \cong M_n(\C)$. However,
by Proposition \ref{M_nAmod} we have 
$\Li_{M_n(\C)}(\C^n)\cong \Li_{\C}(\C) \cong \C$. 
\end{remark}

\begin{proposition}\label{L_AX}
Let $\X$ be the right Hilbert $A$-module described in Example \ref{mainE} above. 
Suppose that $\X\Hi_0$
is dense in $\Hi_1$. Define $B \subseteq \Li(\Hi_1)$ by 
\[
B=\left\{ b \in \Li(\Hi_1) \colon bx, b^*x \in \X \mbox{ for all } x \in \X)  \right\}.
\]
For each $b \in B$ we get a map $\tau(b) \colon \X \to \X$, given by $\tau(b)(x)=bx$. 
Then, $B$ is $*$-isomorphic to $\mathcal{L}_A(\X)$, via the map that sends $b \in B$
 to $\tau(b)$.
\end{proposition}
\begin{proof}
For any $b \in B$
and any $x,y \in \X$, we have 
\[
\langle bx, y \rangle_A = (bx)^*y = x^*(b^*y) = \langle x, b^*y \rangle_A.
\]
Thus, $\tau(b) \in \Li_A(\X)$ and $\tau(b)^*=\tau(b^*)$. It is also easily checked 
that $\tau$ is $*$-homomorphism. Furthermore, it follows from density of $\X\Hi_0$ in $\Hi_1$ 
that $\tau$ is also injective. 

We will finish the proof if we show that $\tau$ is surjective.
Take any $t \in \Li_A(\X)$. We have maps $t\colon \X \to \X$ and $t^*\colon \X \to \X$ satisfying 
\begin{equation}\label{innerX}
t(x)^*y=x^*t^*(y)
\end{equation}
for all $x, y \in \X$. Thus, if $n \in \Z_{\geq 1}$, $x_1, \ldots, x_n \in \X$, and 
$\xi_1, \ldots, \xi_n \in \Hi_0$, then we find, using (\ref{innerX}) at the final step, 
\begin{equation}\label{normSum}
\Bigl\| \sum_{j=1}^n t(x_j)\xi_j \Bigr\|^2 = 
\sum_{j=1}^n\sum_{i=1}^n \bigl\langle t(x_j)\xi_j, t(x_i)\xi_i\bigr\rangle
=  \sum_{j=1}^n\sum_{i=1}^n \bigl\langle x_i^*(t^*t)(x_j)\xi_j, \xi_i\bigr\rangle.
\end{equation}
Recall from Proposition \ref{M_nAmod} that $\X^n$ can be viewed as a 
Hilbert $M_n(A)$-module and 
that $\Li_A(\X) \cong \Li_{M_n(A)}(\X^n) $ via the map $t \mapsto \kappa(t)$. 
Applying Lemma \ref{tx_tx} to $\kappa(t) \in \Li_{M_n(A)}(\X^n)$, we get 
\begin{equation}\label{key_ineq}
(x_i^*(t^*t)(x_j))_{i,j} = (\langle t(x_i), t(x_j) \rangle_A)_{i,j} \leq \|\kappa(t)\|^2 (\langle x_i, x_j \rangle_A)_{i,j}=\|t\|^2 (x_i^*x_j)_{i,j}.
\end{equation}
Therefore, if we let $\bm{\xi}=(\xi_1, \ldots, \xi_n) \in \Hi_0^n$, and 
consider the obvious action of $M_n(A)$ on $\Hi_0^n$, then we get, using 
(\ref{key_ineq}) at the second step,
\begin{align*}
\sum_{j=1}^n\sum_{i=1}^n \bigl\langle x_i^*(t^*t)(x_j)\xi_j, \xi_i\bigr\rangle 
 & = \bigl \langle (x_i^*(t^*t)(x_j))_{i,j}\bm{\xi}, \bm{\xi} \bigr\rangle \\
& \leq \|t\|^2 \bigl \langle (x_i^*x_j)_{i,j}\bm{\xi}, \bm{\xi} \bigr\rangle \\
& = \|t\|^2 \sum_{j=1}^n\sum_{i=1}^n \langle x_j\xi_j, x_i\xi_i \rangle \\
& = \|t\|^2\Bigl\| \sum_{j=1}^n x_j\xi_j \Bigr\|^2.
\end{align*}
This, together with (\ref{normSum}), shows that
\begin{equation}\label{welldefb_t}
\Bigl\| \sum_{j=1}^n t(x_j)\xi_j \Bigr\| \leq \|t\|\Bigl\| \sum_{j=1}^n x_j\xi_j \Bigr\|.
\end{equation}
We can now define $b_t\colon \X\Hi_0 \to \Hi_1$ by letting $b_t(x \xi)=t(x)\xi \in \Hi_1$, 
and extending linearly to all of $\X\Hi_0$. That is, 
\[
b_t\Bigl(\sum_{j=1}^nx_j\xi_j\Bigr) = \sum_{j=1}^n t(x_j)\xi_j.
\] 
Notice that (\ref{welldefb_t}) shows that $b_t$ is well 
defined and that $\|b_t(\eta)\|\leq \|t\|\|\eta\|$, for all 
$\eta \in \X\Hi_0=\op{span}\{ x\xi \colon x \in \X \mbox{ and } \xi \in \Hi_0\}$. 
Thus, we extend $b_t$ by continuity to all of $\Hi_1$,
 and get a well defined map $b_t \in \Li(\Hi_1)$ such that 
 $\|b_t(\eta)\| \leq \|t\| \|\eta \|$ for all $\eta \in \Hi_1$. 
Let $x \in \X$. Since for all $\xi \in \Hi_0$, we have
$(b_tx)\xi = b_t(x\xi)=t(x)\xi$, it follows that 
$b_tx=t(x) \in \X$. Similarly, for any $x,y \in \X$,  we have
 $x^*t(y)=x^*b_ty=(b_t^*x)^*y$ and therefore
 \[
 \langle t^*(x), y\rangle_A =\langle x,t(y)\rangle_A= x^*t(y)
 =(b_t^*x)^*y = \langle b_t^*x, y \rangle_A.
 \]
Hence,  $b_t^*x=t^*(x) \in \X$. 
Thus, $b_t \in B$ and since $\tau(b_t)(x)=b_tx=t(x)$, 
 surjectivity of $\tau$ now follows, finishing the proof. 
\end{proof}

\section{Representations of Hilbert Bimodules and Hilbert Modules}\label{s2}

The main purpose of this section is to state known results 
for representations of Hilbert modules and bimodules. 
Our main contribution here is Proposition \ref{RepforKandL}, where we use Proposition \ref{K_AX} and 
Proposition \ref{L_AX} above to characterize the adjointable and 
the compact-module maps for a representation 
of a right Hilbert module. Such representations are guaranteed to exist
by Corollary \ref{XisMainE} below.
Roughly, this corollary states that for any right Hilbert 
$A$-module $\X$, there are Hilbert spaces $\Hi_0$ and $\Hi_1$ 
and an isometric linear map $\pi_\X \colon \X \to \Li(\Hi_0, \Hi_1)$ 
such that $\pi_\X(\X)$ has the right Hilbert module structure from Example 
\ref{mainE} above. First, we need to recall the concept of Hilbert bimodules
and their representations

\begin{definition}\label{Defbimod}
Let $A$ and $B$ be C*-algebras. A Hilbert $A$-$B$-bimodule is a complex 
vector space $\X$ that is a left Hilbert $A$-module and a right Hilbert $B$-module
(see Notation \ref{HMconv})
such that for all $x,y, z \in \X$,
\begin{equation}\label{bimod}
{}_A\langle x,y \rangle z = x \langle y , z \rangle_B.
\end{equation}
\end{definition}

\begin{remark}\label{AadjX}
For the definition of Hilbert $A$-$B$-bimodule, some authors also require that $A$ acts
on $\X$ via $\langle - , - \rangle_B$-adjointable operators and $B$ acts on $\X$ via 
$_{A}\langle - , - \rangle$-adjointable operators. That is, 
for all $a \in A$, $b \in B$, and $x,y \in \X$ the following holds
$\langle ax,y \rangle_B=\langle x, a^*y\rangle_B$ and ${}_A\langle x,yb \rangle={}_A\langle xb^*,y \rangle$. 
However, this is redundant as it already follows from (\ref{bimod}); see comments after Remark 1.9 in \cite{BMS1994}.
\end{remark}

We now turn to representations of Hilbert bimodules. The following comes mostly from Definition 4.5 
in \cite{Exel1993}.

\begin{definition}\label{RepBimodDef}
Let $\X$ be a Hilbert $A$-$B$-bimodule.
A \emph{representation of $\X$ on a pair of Hilbert spaces $(\Hi_0, \Hi_1)$} 
consists of a triple 
$(\lambda_A, \rho_B, \pi_\X)$,  where $\lambda_A$ is a representation of $A$ 
on $\Hi_1$, $\rho_B$ is a representation of $B$ on 
$\Hi_0$, and $\pi_\X \colon \X \to \Li(\Hi_0, \Hi_1)$ is 
a linear map, such that for all $a \in A, b \in B$, and $x,y \in \X$,
the following compatibility conditions are satisfied.
\begin{enumerate}
\item $\pi_\X(ax)=\lambda_A(a)\pi_\X(x)$, \label{cc1}
\item $\pi_\X(xb)=\pi_\X(x)\rho_B(b)$, \label{cc2}
\item $\lambda_A({}_A\langle x,y \rangle)=\pi_\X(x)\pi_\X(y)^*$, \label{cc3}
\item $\rho_B(\langle x,y \rangle_B)=\pi_\X(x)^*\pi_\X(y)$. \label{cc4}
\end{enumerate} 
If $\pi_\X$ is an isometry, we say the representation $(\lambda_A, \rho_B, \pi_\X)$ is \emph{isometric}. 
\end{definition}

\begin{remark}\label{FaithIsom}
The map $\pi_\X$ in Definition \ref{RepBimodDef} is required to be bounded 
in Definition 4.5 
in \cite{Exel1993}. However, since both $\lambda_A$ and $\rho_B$ are 
$*$-homomorphisms, boundedness of $\pi_\X$ 
follows either from compatibility condition (\ref{cc3}) or (\ref{cc4}). Indeed, for 
instance, compatibility condition (\ref{cc3}) gives
\[
\| \pi_X(x) \|^2 = \| \pi_\X(x)\pi_\X(x)^* \| = \| \lambda_A({}_A\langle x,x \rangle) \|
 \leq \| {}_A\langle x,x \rangle \| = \|x\|^2.
\]
Similarly, Proposition 4.6 in \cite{Exel1993} shows that $(\lambda_A, \rho_B, \pi_\X)$ 
is an isometric representation of a Hilbert $A$-$B$-bimodule $\X$ 
whenever
either $\lambda_A$ or $\rho_B$ is faithful. Indeed, for example, if 
$\rho_B$ is isometric, then by the compatibility condition (\ref{cc4}) we have
\[
\|\pi_\X(x)\|^2 = \|\pi_\X(x)^*\pi_\X(x)\|=\|\rho_B(\langle x,x \rangle_B)\| 
= \| \langle x,x\rangle_B \| = \| x \|^2.
\]
It is also worth mentioning that conditions (\ref{cc1}) and (\ref{cc2}) are actually 
redundant for they respectively follow from  conditions (\ref{cc3}) and (\ref{cc4}).
See Remark 2.2.7 in \cite{Del2023} for the details. 
\end{remark}

The following theorem establishes the existence of representations for any 
Hilbert $A$-$B$-bimodule.

\begin{theorem}\label{BimodRep}
Let $A$ and $B$ be C*-algebras, and let $\X$ be a Hilbert 
$A$-$B$-bimodule. Then, for any nondegenerate representation $\rho_B$ of $B$ on a Hilbert space $\Hi_0$, 
there are a nondegenerate representation $\lambda_A\colon A \to \Li(\Hi_1)$ of $A$ on a Hilbert space 
$\Hi_1$ and a bounded linear map $\pi_\X \colon \X \to \Li(\Hi_0, \Hi_1)$,
such that $(\lambda_A, \rho_B, \pi_\X)$ is a representation of $\X$ on 
$(\Hi_0, \Hi_1)$.
\end{theorem}
\begin{proof}
See Proposition 4.7 in \cite{Exel1993}.
\end{proof}

\begin{corollary}\label{RepBimodIsom}
Let $A$ and $B$ be C*-algebras, and let $\X$ be a Hilbert $A$-$B$-bimodule.
Then there is an isometric representation $(\lambda_A, \rho_B, \pi_\X)$ of $\X$ on 
some pair of Hilbert spaces $(\Hi_0, \Hi_1)$. 
\end{corollary}
\begin{proof}
Let $\rho_B \colon B \to \Li(\Hi_0)$ be the universal representation of $B$. 
Then, $\rho_B$ is faithful and nondegenerate. Hence, this follows at once from 
Theorem \ref{BimodRep} and Remark \ref{FaithIsom}. 
\end{proof}

We now present the definition for a representation of a right Hilbert module, 
which comes 
from looking at the conditions in Definition \ref{RepBimodDef} that
only deal with the right action and right inner product. 

\begin{definition}\label{RepMod}
Let $A$ be a C*-algebra and let $\X$ be a right Hilbert $A$-module.
A \emph{representation of $\X$ on a pair of Hilbert spaces $(\Hi_0, \Hi_1)$}
consists of a pair $(\rho_A, \pi_\X)$
such that $\rho_A$ is a representation of $A$ 
on $\Hi_0$, and $\pi_\X \colon \X \to \Li(\Hi_0, \Hi_1)$ 
is a linear map, such that for all $a \in A$, and all $x,y \in \X$,
the following compatibility conditions are satisfied.
\begin{enumerate}
\item $\pi_\X(xa)=\pi_\X(x)\rho_A(a)$, \label{cm1}
\item $\rho_A(\langle x,y \rangle_A)=\pi_X(x)^*\pi_X(y)$. \label{cm2}
\end{enumerate} 
If $\pi_\X$ is an isometry, we say the representation $(\rho_A, \pi_\X)$ is \emph{isometric}. 
\end{definition}

The map $\pi_\X$ in Definition \ref{RepMod} is always bounded
and this follows exactly as in Remark \ref{FaithIsom}. Similarly, 
faithfulness of $\rho_A$ is sufficient for  
$(\rho_A, \pi_\X)$ to be isometric. Also, as mentioned 
by the end of Remark \ref{FaithIsom}, we point out that condition (\ref{cm1}) in 
Definition \ref{RepMod} is actually implied by  condition (\ref{cm2}).

The following result establishes the existence of (isometric) representations for right Hilbert modules. 
\begin{corollary}\label{XisMainE}
Let $A$ be a C*-algebra and let $\X$ be a right Hilbert $A$-module. Then, for 
any nondegenerate representation $\rho_A$ of $A$ on a Hilbert space $\Hi_0$,
there are a Hilbert space $\Hi_1$ and a linear map $\pi_\X \colon \X \to \Li(\Hi_0, \Hi_1)$
such that $(\rho_A, \pi_\X)$ is a representation of $\X$ on $(\Hi_0, \Hi_1)$ as in Definition \ref{RepMod}. 
Furthermore, if $\rho_A$ is faithful, then $(\rho_A, \pi_\X)$ is isometric and in this case
$\pi_\X(\X)$ has the right Hilbert $\rho_A(A)$-module structure from Example 
\ref{mainE}.
\end{corollary}
\begin{proof}
It is well known that a right Hilbert $A$-module $\X$ is also a Hilbert $\mathcal{K}_A(\X)$-$A$-bimodule. 
Hence the desired result follows at once from Theorem \ref{BimodRep}. 
The isometric part of the statement follows from Remark \ref{FaithIsom}.
\end{proof}
\begin{remark}
There are at least two different approaches in the current literature
to prove Corollary \ref{XisMainE} that 
do not depend on Theorem \ref{BimodRep}. Indeed, the first one is to take
$\pi_{\X}$ to be the restriction to $\X$ of the map $U$ from Theorem 2.6 in \cite{Zettl1983}. 
We are thankful to Julian Kranz for pointing out this reference to us. 
The second one, uses Murphy's theory on positive definite kernels for Hilbert modules developed in \cite{murphy_1997}, in which $\pi_\X$ comes from a Kolmogorov decomposition 
of the positive definite map $\X \times \X \to \Li(\Hi_0)$ given by $(x,y) \mapsto \rho_A(\langle x, y \rangle_A)$. 
The details of this approach are in Theorem 3.1 in \cite{murphy_1997}, we are thankful to the 
referee for pointing out this result. 
\end{remark}

We end this section by observing that our main results from Section \ref{s1} 
can be stated using the language of Definition \ref{RepMod}.

\begin{proposition}\label{RepforKandL}
Let $A$ be a C*-algebra, let $\X$ be any right Hilbert $A$-module, 
and let $(\rho_A, \pi_\X)$ be a representation of $\X$ 
on $(\Hi_0, \Hi_1)$, with $\rho_A$ faithful.  
Suppose that $\pi_\X(\X)\Hi_0$ is dense in $\Hi_1$. Then, the 
C*-algebras  $\mathcal{K}_A(\X)$ and $\mathcal{L}_A(\X)$ can be represented 
on $\Hi_1$ via the maps described below. 
\begin{enumerate}
\item There is a $*$-isomorphism from $\mathcal{K}_A(\X)$ to
\[
 \cj{\op{span} \left\{ \pi_\X(x)\pi_\X(y)^* \colon x,y \in \X \right\}} \subseteq \Li(\Hi_1),
\] 
which sends $\theta_{x,y}$ to $\pi_\X(x)\pi_\X(y)^*$ for $x,y \in \X$.\label{RCL1}
\item We define  
\[
B=\left\{ b \in \Li(\Hi_1) \colon b\pi_\X(x), b^*\pi_\X(x) \in \pi_\X(\X)  \mbox{ for all }  x \in \X \right\}.
\]
For each $b \in B$ we get a map $\tau(b) \colon \pi_\X(\X) \to \pi_\X(\X)$, 
which sends $\pi_\X(x)$ to $\tau(b)(\pi_\X(x))=b\pi_\X(x)$. 
Then, $B$ is $*$-isomorphic to $\mathcal{L}_A(\X)$, via the map that sends $b \in B$
 to $\pi_\X^{-1} \circ \tau(b)\circ \pi_\X$, where $\pi_\X^{-1}$ 
 is interpreted as the inverse of the linear bijection $\pi_\X \colon \X \to \pi_\X(\X)$. \label{RCL2}
\end{enumerate}
\end{proposition}
\begin{proof}
Since $\rho_A$ is faithful, $\pi_\X$ is isometric. The result now 
follows immediately after replacing $A$ with its isometric copy $\rho_A(A)$ 
and $\X$ with its isometric copy $\pi_\X(\X)$ on  
Proposition \ref{K_AX} for part (\ref{RCL1}), and on Proposition \ref{L_AX} for part (\ref{RCL2}). 
\end{proof}

\section{Representations of C*-Correspondences}\label{s3}

In this section we define representations of C*-correspondences and 
present the main result of this paper, Theorem \ref{CorrespRep}, 
which we will see is actually a generalization of Theorem \ref{BimodRep}. 
We then give two applications of this theorem. 
The first, contained in Theorem \ref{Suff_Nec_Cond_Corres_is_Bimod} and Corollary \ref{COR_Suff_Nec_Cond_Corres_is_Bimod}, 
gives necessary and sufficient conditions for a general $(A,B)$ C*-correspondence 
to admit a Hilbert $A$-$B$-bimodule structure. The second, given in Theorem \ref{TensorCorres}, 
shows that the interior tensor product of correspondences admits a representation 
as the product of suitable representations of the factors. 

\subsection{Definitions}

We start by recalling the definition of a C*-correspondence. 

\begin{definition}
Let $A$ and $B$ be C*-algebras. An \emph{$(A,B)$ C*-correspondence} is a pair $(\X, \varphi)$, where
$\X$ is a right Hilbert $B$-module and $\varphi \colon A \to \Li_B(\X)$ is a $*$-homomorphism.
We say that $A$ acts \emph{nondegenerately} on $\X$ whenever $\varphi(A)\X$ is dense in $\X$.
\end{definition}

We observe that Remark \ref{AadjX} implies that any Hilbert $A$-$B$-bimodule, as in Definition \ref{Defbimod}, is in fact an 
$(A,B)$ C*-correspondence with $\varphi$ given by the left action
of the bimodule. In fact, if $\X$ is  Hilbert $A$-$B$-bimodule, 
then it is well known that $A$ always acts nondegenerately on $\X$. 
However, not every C*-correspondence is a Hilbert bimodule. 
Thus, C*-correspondences are a generalization of Hilbert bimodules. 
Our goal is then to find a general version of Theorem \ref{BimodRep}
that works for the general C*-correspondence setting. For this,
we need first to define what we mean by representations of C*-correspondences. 

\begin{definition}\label{DefCorrespRep}
Let $A$ and $B$ be C*-algebras, and let $(\X, \varphi)$ be an $(A,B)$ C*-correspondence. 
A \emph{representation of $(\X, \varphi)$ on a pair of Hilbert spaces $(\Hi_0, \Hi_1)$} consists 
of a triple $(\lambda_A, \rho_B, \pi_\X)$ where $\lambda_A$ is 
a representation of $A$ on $\Hi_1$,  $\rho_B$ is a representation of $B$ on $\Hi_0$, and 
 $\pi_\X \colon \X \to \Li(\Hi_0, \Hi_1)$ is a linear map,
such that for all $a \in A$, and all $x,y \in \X$, the following 
compatibility conditions are satisfied.
\begin{enumerate}
\item $\pi_\X(\varphi(a)x)=\lambda_A(a)\pi_\X(x)$, \label{cco1}
\item $\pi_\X(xb)=\pi_\X(x)\rho_B(b)$, \label{cco2}
\item $\rho_B(\langle x,y \rangle_B)=\pi_\X(x)^*\pi_\X(y)$. \label{cco3}
\end{enumerate} 
If $\pi_\X$ is an isometry, we say the representation $(\lambda_A, \rho_B, \pi_\X)$ is \emph{isometric}. 
\end{definition}

As in Remark \ref{FaithIsom}, the linear map $\pi_\X$
from Definition \ref{DefCorrespRep} is automatically bounded and 
faithfulness of $\rho_B$ is sufficient for  
$(\lambda_A, \rho_B, \pi_\X)$ to be isometric. 
Similarly, condition (\ref{cco2}) in Definition \ref{DefCorrespRep}
is automatically implied by condition (\ref{cco3}).

We point out that Definition \ref{DefCorrespRep} agrees with the 
definitions of representations of C*-correspondences in the literature. Indeed, 
suppose that $(\X, \varphi)$ is an $(A, A)$ C*-correspondence and that $(\lambda_A, \rho_A, \pi_\X)$
is a representation of  $(\X, \varphi)$ as in Definition \ref{DefCorrespRep} with $\Hi_0=\Hi_1$ and $\lambda_A = \rho_A$.
Then, $(\lambda_A, \pi_\X)$
is a representation of  $(\X, \varphi)$ on $\Li( \Hi_0)$ in the sense of Definition 2.1 in \cite{kat2004} 
and an isometric covariant representation of  $(\X, \varphi)$ on $\Hi_0$ in the sense of Definition 2.11 in \cite{MuhlySolel1998}. 
More generally, suppose that $(\X, \varphi)$ is an $(A, B)$ C*-correspondence and that $(\lambda_A, \rho_B, \pi_\X)$
is a representation of  $(\X, \varphi)$ on $(\Hi_0, \Hi_1)$ as in Definition \ref{DefCorrespRep}. 
Then, letting $C= \Li(\Hi_0 \oplus \Hi_1)$, we get obvious maps $\widehat{\lambda_A}\colon A \to C$,  $\widehat{\rho_B}\colon B \to C$, and  $\widehat{\pi_\X}\colon  \X \to C$ induced by $\lambda_A, \rho_B$, and $\pi_\X$. 
It is clear that $(\widehat{\lambda_A}, \widehat{\rho_B}, \widehat{\pi_\X})$ is a rigged representation of $(\X, \varphi)$ on $C$ in the sense of Definition 3.7 in \cite{CDE2022CK}.

\subsection{Interior tensor product of C*-correspondences.}

For the rest of this section we will need the interior tensor product of C*-correspondences.
This is a well known construction.
We only list some of the basic properties that will be needed below. 
We refer the reader to Proposition 4.5 in \cite{lance_1995} and
the afterwards discussion for more details.
Let $A$, $B$, and $C$ be C*-algebras, let  $(\X, \varphi_\X)$ 
be an $(A,B)$ C*-correspondence and let $(\Y, \varphi_\Y)$
be a $(B, C)$ C*-correspondence. We consider 
 the algebraic $B$-balanced tensor product of modules $\X \odot_B \Y$
 which has a $C$-valued right pre-inner product given
 on elementary tensors by
\begin{equation}\label{innerProd}
\langle x_1 \otimes y_1, x_2 \otimes y_2 \rangle_{C} =
 \langle y_1, \varphi_\Y(\langle x_1, x_2 \rangle_B)y_2\rangle_C.
\end{equation}
The completion of $\X \odot_B \Y$ under the norm induced by 
the $C$-valued right pre-inner product from equation (\ref{innerProd}) is a right Hilbert $C$-module, 
which we denote by $\X \otimes_{\varphi_\Y} \Y$. 
It is useful to keep in mind that, 
by construction, if $x \in \X$, $b \in B$, and $y \in \Y$, then 
\begin{equation}\label{modaction}
xb \otimes y = x \otimes \varphi_\Y(b)y.
\end{equation}
Furthermore, $A$ acts on $\X \otimes_{\varphi_\Y} \Y$
via $\langle - , -\rangle_{C}$-adjointable operators 
and the action $\widetilde{\varphi_\X} \colon  A \to \Li_C(\X \otimes_{\varphi_\Y} \Y)$
is determined by $\varphi_\X$
as follows:
\begin{equation}\label{leftactionTensor}
\widetilde{\varphi_\X}(a)(x \otimes y) = \varphi_\X(a)x \otimes y,
\end{equation}
for $a \in A$, $x \in \X$, and $y \in \Y$. 
All this makes $(\X \otimes_{\varphi_\Y}\Y, \widetilde{\varphi_\X})$ into 
an $(A,C)$ C*-correspondence, called the \emph{interior tensor product}
of $(\X, \varphi_\X)$ with $(\Y, \varphi_\Y)$. 

\subsection{Main results.}
For any two C*-algebras $A$ and $B$, the following result
shows how to produce a representation of any 
$(A,B)$ C*-correspondence out of any nondegenerate representation of $B$. 

\begin{theorem}\label{CorrespRep}
Let $A$ and $B$ be C*-algebras and let $(\X, \varphi)$ be an ($A$,$B$) C*- correspondence.
Then, for any nondegenerate representation $\rho_B$ of $B$ on a Hilbert space $\Hi_0$, 
there are a representation $\lambda_A\colon A \to \Li(\Hi_1)$ of $A$ on a Hilbert space 
$\Hi_1$ and a bounded linear map $\pi_\X \colon \X \to \Li(\Hi_0, \Hi_1)$,
such that $(\lambda_A, \rho_B, \pi_\X)$ is a representation of $(\X, \varphi)$ on 
$(\Hi_0, \Hi_1)$
as in Definition \ref{DefCorrespRep}. 
If in addition $A$ acts nondegenerately on $\X$,
then $\lambda_A$ is nondegenerate. 
\end{theorem}

\begin{proof}
Notice that $(\Hi_0, \rho_B)$ is a $(B, \C)$ C*-correspondence. Let
$\Hi_1=\X \otimes_{\rho_B} \Hi_0$ be the interior tensor product of $(\X, \varphi)$ 
with $(\Hi_0, \rho_B)$, which is 
in particular a right Hilbert $\C$-module, that is, 
a Hilbert space. 
The representation of $A$ on $\Hi_1$ comes 
from the left action of $A$ on $\Hi_1$ 
gotten from  equation (\ref{leftactionTensor}) in the 
interior tensor product construction. 
Indeed,  Proposition 2.66 
of \cite{raeburn_williams_1998} gives
$\lambda_{A}\colon A \to \Li(\Hi_1)$, a representation 
of $A$, such that for each $a \in A$, 
each $x \in \X$, and each $\xi \in \Hi_0$,
\begin{equation}\label{pi_A_phi_X}
\lambda_{A}(a)(x \otimes \xi)=\varphi(a)x \otimes \xi.
\end{equation}
Furthermore, it is also shown in Proposition 2.66 
of \cite{raeburn_williams_1998} that 
$\lambda_A$ is nondegenerate whenever 
$A$ acts nondegenerately on $\X$. 
We now establish the existence of $\pi_\X$. This is motivated by 
the Fock space construction in \cite{pim1997}.
Indeed, for each $x \in \X$, let $\pi_\X(x)\colon \Hi_0 \to \Hi_1$ be the creation operator 
\begin{equation}\label{PI_X}
\pi_\X(x)\xi=x \otimes \xi.
\end{equation}
Then, it is clear that $x \mapsto \pi_\X(x)$ is a linear map
from $\X$ to $\Li(\Hi_0, \Hi_1)$. As in Remark \ref{FaithIsom},
boundedness of $\pi_\X$ will follow once we check the 
compatibility conditions from Definition \ref{DefCorrespRep},
which will in turn prove that $(\lambda_A, \rho_B, \pi_\X)$ is indeed a representation 
of $(\X, \varphi)$ on $(\Hi_0, \Hi_1)$. 
First we check condition (\ref{cco1}). If $a \in A$, $x \in \X$, and $\xi \in \Hi_0$, then 
\[
\pi_\X(\varphi(a)x)\xi = (\varphi(a)x)\otimes \xi = \lambda_A(a)(x \otimes \xi) 
= \lambda_A(a)\pi_\X(x)\xi.
\]
That is, $\pi_\X(\varphi(a)x)=\lambda_A(a)\pi_\X(x)$ as desired. Now for  
$b \in B$, $x \in \X$, and $\xi \in \Hi_0$, we use equation (\ref{modaction})
at the second step and find
\[
\pi_\X(xb)\xi=(xb) \otimes \xi = x \otimes \rho_B(b)\xi = \pi_\X(x)\rho_B(b)\xi.
\]
This gives that $\pi_\X(xb)=\pi_\X(x)\rho_B(b)$, proving condition  (\ref{cco2}). 
Finally, notice that equation (\ref{innerProd}) shows that 
$\pi_\X(x)^*\colon \Hi_1 \to \Hi_0$ is the 
annihilation operator satisfying, for any $z \in \X$ and $\xi \in \Hi_0$, 
\begin{equation}\label{aanih}
\pi_\X(x)^*(z \otimes \xi) = \rho_B(\langle x,z\rangle_B)\xi.
\end{equation}
Thus, for any $\xi \in \Hi_0$,
\[
\pi_\X(x)^*\pi_\X(y)\xi =\pi_\X(x)^*(y\otimes \xi) =\rho_B(\langle x, y \rangle_B)\xi,
\]
whence $\pi_\X(x)^*\pi_\X(y)  = \rho_B(\langle x, y \rangle_B)$, 
which is compatibility condition (\ref{cco3}), so we are done.
\end{proof}

The method we used in the proof of Theorem \ref{CorrespRep} can be easily adapted to produce 
a different proof of Theorem \ref{BimodRep}. Thus, we present below a restatement of Theorem \ref{BimodRep}
followed by a proof along the lines of the proof of Theorem \ref{CorrespRep}. 

\begin{theorem}\label{BimodRepInCorrespRep}
Let $A$ and $B$ be C*-algebras, and let $\X$ be a Hilbert 
$A$-$B$-bimodule. Then, for any nondegenerate representation $\rho_B$ of $B$ on a Hilbert space $\Hi_0$, 
there are a nondegenerate representation $\lambda_A\colon A \to \Li(\Hi_1)$ of $A$ on a Hilbert space 
$\Hi_1$ and a bounded linear map $\pi_\X \colon \X \to \Li(\Hi_0, \Hi_1)$,
such that $(\lambda_A, \rho_B, \pi_\X)$ is a representation of $\X$ on 
$(\Hi_0, \Hi_1)$
as in Definition \ref{RepBimodDef}. 
\end{theorem}
\begin{proof}
We get the Hilbert space $\Hi_1=\X \otimes_{\rho_B} \Hi_0$ and the map $\pi_\X \colon \X \to \Li(\Hi_0, \Hi_1)$ exactly as in the proof of Theorem \ref{CorrespRep}. Since 
$A$ acts on $\X$ via $\langle -, - \rangle_B$-adjointable operators (see Remark \ref{AadjX}),
 we use Proposition 2.66 
of \cite{raeburn_williams_1998} to get 
$\lambda_{A}\colon A \to \Li(\Hi_1)$, a representation 
of $A$, such that for each $a \in A$, 
each $x \in \X$, and each $\xi \in \Hi_0$,
\[
\lambda_{A}(a)(x \otimes \xi)=(ax) \otimes \xi.
\]
Furthermore, since $\X$ is a left Hilbert $A$-module, it follows that $A$ acts nondegenerately on $\X$.
Thus, Proposition 2.66 
of \cite{raeburn_williams_1998} also guarantees that $\lambda_{A}$ is nondegenerate. 
Finally, compatibility condition (\ref{cc1}) from Definition \ref{RepBimodDef} 
is shown exactly as compatibility condition (\ref{cco1}) from Definition \ref{DefCorrespRep} was 
shown in the proof of Theorem \ref{CorrespRep}. Since 
compatibility conditions (\ref{cc2}) and (\ref{cc4}) from Definition \ref{RepBimodDef} 
coincide with compatibility conditions (\ref{cco2}) and (\ref{cco3})
from Definition \ref{DefCorrespRep}, we only need to make sure that compatibility condition (\ref{cc3}) of Definition 
\ref{RepBimodDef} is satisfied. Indeed, for any $x,y,z \in \X$, and any $\xi \in \Hi_0$,
using equation (\ref{aanih}) at the first step, equation (\ref{modaction}) at the second step,
and equation (\ref{bimod})
at the third one, we get
\begin{align*}
\pi_\X(x)\pi_\X(y)^*(z \otimes \xi) & =x \otimes \rho_B(\langle y,z\rangle_B)\xi \\
& = x\langle y,z\rangle_B \otimes \xi \\
& = {}_A\langle x, y \rangle z \otimes \xi \\
&= \lambda_A({}_A\langle x, y \rangle)(z \otimes \xi).
\end{align*}
Thus, $\pi_\X(x)\pi_\X(y)^*=\lambda_A({}_A\langle x, y \rangle)$, as wanted. 
\end{proof}

We give three remarks about the last two results. On those we
explain how these results are useful to apply 
the main result of Section \ref{s2} and also how the proofs of 
these theorems compare to those known in the current literature.

\begin{remark}\label{Daws_Comment4}
The construction of the Hilbert space $\Hi_1$ given in the proofs 
of Theorems \ref{CorrespRep} and \ref{BimodRepInCorrespRep}
above clearly implies that $\pi_X(\X)\Hi_0$ is dense in $\Hi_1$. 
This is also true for the Hilbert space $\Hi_1$ obtained from 
Theorem \ref{BimodRep}. 
Indeed, according to the proof of Proposition 4.7 in \cite{Exel1993}, 
the space $\Hi_1$ is defined as follows. Let $L$ be the linking algebra 
of $\X$ and $\iota \colon \X \to L$ the inclusion of $\X$ 
as the upper right corner of $L$. 
Then, $\Hi_1$ is defined as the closure of $\pi(
\iota(\X))\Hi_0$,
where $\pi$ is a suitable representation of $L$ on a Hilbert space $\Hi$ that contains $\Hi_0$.
For each $x \in \X$, the operator $\pi_\X(x) \in \Li(\Hi_0, \Hi_1)$
is then defined as the restriction of $\pi(\iota(x))$ to $\Hi_0$,
so it follows that $\pi_\X(\X)\Hi_0$ is dense in $\Hi_1$. 
This shows that the map $\pi_\X$, no matter from 
which construction presented so far was obtained, 
satisfies the nondegeneracy condition on the hypothesis 
of Proposition \ref{RepforKandL}. 
\end{remark}

\begin{remark}
The proof of Theorem \ref{BimodRep} given in \cite{Exel1993}
does not appear to have an obvious modification to
make it work for the C*-correspondence case. This is
due to the fact that their proof relies on the linking algebra of
the bimodule $\X$, which does not exist in the general C*-correspondence 
setting due to the lack of an $A$-valued left inner product.
We also believe that the arguments used in \cite{Zettl1983}
to prove our Corollary \ref{XisMainE}
can't be modified to produce a proof of Theorem \ref{CorrespRep}. Finally, we point out that 
the methods we employed to show Theorem \ref{CorrespRep}
differ from those used in \cite{Zettl1983} and \cite{Exel1993}. 
In particular, we have obtained in equation (\ref{PI_X}) a concise formula for $\pi_\X$ that 
might be useful to produce concrete representations 
of both Hilbert bimodules and modules. 
\end{remark}

\begin{remark}\label{repFaith}
It follows from the definition of $\lambda_A$ in \eqref{pi_A_phi_X} 
that if $\varphi$ is not injective, then $\lambda_A$ is not faithful. 
The converse holds provided that the representation $\rho_B$ in the hypotheses of 
Theorem \ref{CorrespRep} (or in Theorem \ref{BimodRepInCorrespRep}) is faithful. 
Indeed, suppose that $\varphi$ is injective and assume, for the sake of 
a contradiction, that there is 
a nonzero $a \in A$ such that $\lambda_{A}(a)=0$. 
Since $\varphi$ is injective, there is a nonzero $x \in \X$ such that $\varphi(a)x\neq 0$
(for the Hilbert bimodule case we interpret $\varphi(a)x$ as $ax$). We can 
find nonzero elements $y \in \X$ and $b \in B$ such that 
$x=yb$ (see for example Proposition 2.31 in \cite{raeburn_williams_1998}). 
Then, for any $\xi \in \Hi_0$, using (\ref{modaction}) at 
the third step, 
\[
0=\lambda_{A}(a)(x \otimes \xi) = \varphi(a)yb \otimes \xi 
= \varphi(a)y \otimes \rho_B(b)\xi.
\]
Since $\varphi(a)x \neq 0$, it follows that $\varphi(a)y \neq 0$ and therefore 
$\rho_B(b)\xi=0$ for all $\xi \in \Hi_0$. Hence, faithfulness of 
$\rho_B$ implies that $b=0$, a contradiction.
\end{remark}

We warn the reader that Remark \ref{repFaith} above was mistakenly stated 
in Remark 3.3.6. in \cite{Del2023} for it required $\varphi$ to be nondegenerate rather than 
injective. We thank Menevse Eryuzlu for pointing this out to us. 

We now present two applications 
of Theorem \ref{CorrespRep}. The first application
answers the problem of determining when 
an $(A,B)$ C*-correspondence can be uniquely given the structure 
of a Hilbert $A$-$B$-bimodule. 

\begin{theorem}\label{Suff_Nec_Cond_Corres_is_Bimod}
Let $A$ and $B$ be C*-algebras, let $(\X, \varphi)$ be an
$(A,B)$ C*-correspondence, let $A_0 = A/\ker(\varphi)$, and  let $\varphi_0 \colon A_0 \to \Li_{B}(\X)$
be the injective $*$-homomorphism induced by $\varphi$, which makes $(\X, \varphi_0)$
an $(A_0,B)$ C*-correspondence.
Then, there is a unique $A_0$-valued left inner product on $\X$
making it a Hilbert $A_0$-$B$-bimodule if and only if $\mathcal{K}_B(\X)
\subseteq \varphi(A)$. 
\end{theorem} 
\begin{proof}
First we establish the uniqueness of the $A_0$-valued left inner product on $\X$.
Let $\rho_B \colon  B \to \Li(\Hi_0)$ 
be the universal representation of $B$, which is 
faithful, and apply Theorem \ref{CorrespRep} to get an isometric representation 
$(\lambda_{A_0}, \rho_B, \pi_{\X})$ for the $(A_0, B)$ C*-correspondence $(\X, \varphi_0)$. Since $\varphi_0$ and $\rho_B$ are injective, Remark \ref{repFaith} shows that $\lambda_{A_0}$ is injective. 
Suppose now that  ${}_{A_0}\langle -,- \rangle \colon \X \times \X \to A_0$ and  ${}_{A_0}(-\mid-) \colon \X \times \X \to A_0$ are two $A_0$-valued left inner products making $\X$ a Hilbert $A_0$-$B$-bimodule. 
The proof of 
Theorem \ref{BimodRepInCorrespRep} now shows that, 
for every $x, y \in \X$,
\[
\lambda_{A_0}({}_{A_0}\langle x,y \rangle) = \pi_\X(x)\pi_\X(y)^*= \lambda_{A_0}({}_{A_0}( x \mid y )).
\]
Since $\lambda_{A_0}$ is injective, the above implies that  ${}_{A_0}\langle x,y \rangle={}_{A_0}( x \mid y )$, as wanted. 

Next, assume that there is $A_0$-valued left inner product on $\X$
making it a Hilbert $A_0$-$B$-bimodule. That is, there is a map 
${}_{A_0}\langle -, -\rangle \colon \X \times \X \to A_0$ such that 
for any $x,y, z \in \X$ we have $\varphi_0({}_{A_0}\langle x, y\rangle)z =x \langle y,z \rangle_B$.
Since $x \langle y,z \rangle_B=\theta_{x,y} (z)$, this proves $\theta_{x,y} \in \varphi_0(A_0) = \varphi(A)$ for any 
$x, y \in \X$, which in turn implies $\mathcal{K}_B(\X)
\subseteq \varphi(A)$. 

Conversely, assume that $\mathcal{K}_B(\X)
\subseteq \varphi(A)$. 
Then $\theta_{x,y} \in \varphi(A)=\varphi_0(A_0)$ for each $x,y \in \X$. Thus, we use the $*$-isomorphism $\varphi_{_0}^{-1} \colon \varphi_0(A_0) \to A_0$ to define 
\begin{equation}\label{L_IP}
{}_{A_0}\langle x, y \rangle =\varphi_{_0}^{-1}(\theta_{x,y}).
\end{equation}
It is immediate to check that $(x,y)\mapsto {}_{A_0}\langle x, y \rangle$ is indeed 
an $A_0$-valued left inner product on $\X$. Furthermore, if $x \in \X$, then 
\[
\| {}_{A_0}\langle x, x \rangle \| = \| \varphi_{_0}^{-1}(\theta_{x,y})\| =\| \theta_{x,x} \| = \| x\|.
\]
Hence, $\X$ is indeed a left Hilbert $A_0$-module. 
Finally, by \eqref{L_IP}, 
\[
\varphi_0({}_{A_0}\langle x, y\rangle)z = \theta_{x,y}z= x \langle y,z \rangle_B,
\]
for all $x,y, z \in \X$, proving equation  (\ref{bimod}) and therefore that 
$\X$ a Hilbert $A_0$-$B$-bimodule, as wanted. 
\end{proof}

\begin{remark}
If the $*$-homomorphism $\varphi \colon A \to \Li_B(\X)$ in  Theorem \ref{Suff_Nec_Cond_Corres_is_Bimod} is injective,
then we get at once that $\mathcal{K}_B(\X)
\subseteq \varphi(A)$ is a necessary and sufficient condition 
for the C*-correspondence $(\X, \varphi)$ to have a unique structure of a Hilbert $A$-$B$-bimodule. 
\end{remark}

\begin{remark}\label{katsuraAA}
In general, if $(\X, \varphi)$ is an $(A,B)$ C*-correspondence where $\varphi$ is 
not necessarily injective, then $(\X, \varphi)$ will have a unique structure of a Hilbert $A$-$B$-bimodule if and only if $A$ has an ideal $J$ such that $J$ is $*$-isomorphic to $\mathcal{K}_B(\X)$ via $\varphi|_{J}$. 
The case $A=B$
was proven by T. Katsura in Lemmas 3.3 and 3.4 from \cite{kat2003}. 
The general case was later shown by A. Buss and R. Meyer in \cite{BussMeyer17}, we are thankful to Ralph Meyer for providing this last
reference to us. 
In Corollary \ref{COR_Suff_Nec_Cond_Corres_is_Bimod} 
we show that the general case (see Lemma 4.2 
in \cite{BussMeyer17}) does follow from Theorem \ref{Suff_Nec_Cond_Corres_is_Bimod}.
\end{remark}

\begin{corollary}\label{COR_Suff_Nec_Cond_Corres_is_Bimod}
Let $A$ and $B$ be C*-algebras and let $(\X, \varphi)$ be an
$(A,B)$ C*-correspondence.
Then, there is a unique $A$-valued left inner product on $\X$
making it a Hilbert $A$-$B$-bimodule if and only if $A$ has in ideal $J$ such that $J$ is $*$-isomorphic to $\mathcal{K}_B(\X)$ via $\varphi|_{J}$. 
\end{corollary}
\begin{proof}
Put $A_0 = A/ \ker(\varphi)$. Assume first that $A$ has an ideal $J$ such that $J$ is $*$-isomorphic to $\mathcal{K}_B(\X)$ via $\varphi|_{J}$. Then $\mathcal{K}_B(\X) = \varphi|_{J}(J) \subseteq \varphi(A)$, so 
Theorem \ref{Suff_Nec_Cond_Corres_is_Bimod} implies that 
$X$ has a unique left $A_0$-valued inner product, 
say ${}_{A_0}\langle -, -\rangle \colon \X \times \X \to A_0$, that makes 
$\X$ into a Hilbert $A_0$-$B$-bimodule. By standard Hilbert module results 
(see for instance Proposition 1.10 in \cite{BMS1994}) the ideal $J_0 = \cj{{}_{A_0}\langle \X, \X\rangle} \subseteq A_0$, the closed span of the $A_0$-valued inner product, is $*$-isomorphic to $\mathcal{K}_B(\X)=\varphi|_{J}(J)$. This implies that 
the left $J_0$-valued inner product can be though as a left $J$-valued inner product, 
which upgrades $\X$ to a Hilbert $A$-$B$-bimodule.

Conversely, if $\X$ is a Hilbert $A$-$B$-bimodule, 
then Proposition 1.10 in \cite{BMS1994} shows that $J_A=\cj{{}_{A}\langle \X, \X\rangle} \subseteq A$,
the closed span of the left $A$-valued inner product, is an ideal in $A$ which is $*$-isomorphic to $\mathcal{K}_B(\X)$ via the restriction of the left action of $A$ to $J_A$. 

Finally, the uniqueness of the left $A$-valued inner product is established 
as in the proof of Theorem \ref{Suff_Nec_Cond_Corres_is_Bimod}, but this time 
applying Theorem \ref{CorrespRep} to the $(J,B)$ C*-correspondence $(\X, \varphi|_{J})$.
\end{proof}

As an application of Theorem \ref{Suff_Nec_Cond_Corres_is_Bimod}, 
we give below an easy proof that a Hilbert space of 
dimension at least 2, thought of as a $(\C, \C)$ C*-correspondence, can't be given the structure of a Hilbert 
$\C$-$\C$-bimodule.

\begin{example}\label{HasCC}
Let $\Hi$ be a Hilbert space with dimension at least $2$.
Clearly $\Hi$ is a right Hilbert $\C$-module,  $\Li_\C(\Hi)=\Li(\Hi)$, 
and $\mathcal{K}_\C(\Hi)=\mathcal{K}(\Hi)$. 
For each $a \in \C$, we define $\varphi(a)=a \cdot \op{id}_{\Hi}$. 
Then, $\varphi \colon \C \to \Li(\Hi)$ makes $(\Hi, \varphi)$ into
a $(\C, \C)$ C*-correspondence with injective left action. Furthermore, 
it's clear that $\varphi(\C)\Hi = \Hi$, so the left action 
is nondegenerate.
We claim that $\mathcal{K}(\Hi) \not\subseteq \varphi(\C)$. 
Let $(\xi_j)_{j \in J}$ be an orthonormal
basis for $\Hi$. 
By assumption $\op{card}(J) \geq 2$
and therefore we can find $j,k \in J$ with $j \neq k$. 
Notice that $\theta_{\xi_j, \xi_k}(\xi_k)=\xi_j$, whence 
$\theta_{\xi_j, \xi_k} \neq \varphi(a)$ for all $a \in \C$, 
proving the claim. Therefore, 
Theorem \ref{Suff_Nec_Cond_Corres_is_Bimod} implies
that $(\X, \varphi)$ can't be given the structure of a 
 Hilbert 
$\C$-$\C$-bimodule. 
\end{example}

\begin{remark}
A direct sum of Hilbert $A$-$B$-bimodules is not,
in general, a Hilbert bimodule again. However, it is an $(A,B)$ C*-correspondence. It is not hard to see that the C*-correspondence in Example 
\ref{HasCC} is a direct sum of Hilbert $\C$-$\C$-bimodules. We have not investigated which C*-correspondences can be decomposed as a direct sum of 
Hilbert bimodules, but we believe Theorem \ref{Suff_Nec_Cond_Corres_is_Bimod} and Corollary \ref{COR_Suff_Nec_Cond_Corres_is_Bimod} might be useful 
tools to tackle this problem. 
\end{remark}

As an application of concrete representations of right Hilbert modules (as in Definition \ref{RepMod}), Murphy 
gives an easy construction of the exterior tensor product of right Hilbert modules (see
Theorem 3.2 in \cite{murphy_1997}). In analogy with Murphy's result, we conclude the paper with a second
application of Theorem \ref{CorrespRep}, which deals
with how to get a representation for the interior 
tensor product of an $(A,B)$ C*-correspondence $(\X, \varphi_\X)$ with a $(B,C)$ 
C*-correspondence $(\Y, \varphi_\Y)$
using particular representations of the C*-correspondences
$(\X, \varphi_\X)$ and $(\Y, \varphi_\Y)$.
The main point is that we can always make 
the representation
of the C*-algebra $B$ from the representation of $(\X, \varphi_\X)$ 
agree with the representation of $B$ from the representation of $(\Y, \varphi_\Y)$.

\begin{theorem}\label{TensorCorres}
Let $A$, $B$, and $C$ be C*-algebras, let $(\X, \varphi_\X)$ 
be an ($A$,$B$) C*-correspondence, let $(\Y, \varphi_\Y)$ be a ($B$,$C$) C*-correspondence
such that $B$ acts nondegenerately on $\Y$, 
and let $\rho_C \colon C \to \Li(\Hi_0)$ be any  nondegenerate representation 
of $C$ on a Hilbert space $\Hi_0$. 
Then:
\begin{enumerate}
\item There are Hilbert spaces $\Hi_1$,$\Hi_2$, maps $\lambda_A \colon A \to \Li(\Hi_2)$, $\sigma_B \colon B \to \Li(\Hi_1)$, $\tau_\X \colon \X \to \Li(\Hi_1, \Hi_2)$, and 
$\pi_\Y \colon \Y \to \Li(\Hi_0, \Hi_1)$, such that 
$(\lambda_A, \sigma_B, \tau_\X)$ is a representation of $(\X, \varphi_\X)$ on $(\Hi_1, \Hi_2)$ and
$(\sigma_B, \rho_C, \pi_\Y)$ is a representation of $(\Y, \varphi_\Y)$ on $(\Hi_0, \Hi_1)$. \label{TC1}
\item For every pair $((\lambda_A, \sigma_B, \tau_\X)$, $(\sigma_B, \rho_C, \pi_\Y))$ as in (\ref{TC1}),
there is a map $\pi \colon \X \otimes_{\varphi_\Y} \Y \to \Li(\Hi_0, \Hi_2)$
satisfying $\pi(x \otimes y)=\tau_\X(x)\pi_\Y(y)$ and 
such that the triple $(\lambda_A, \rho_C, \pi)$ is a representation of 
$(\X \otimes_{\varphi_\Y} \Y, \widetilde{\varphi_\X})$ on $(\Hi_0, \Hi_2)$. \label{TC2}
\item The map $\pi$ from (\ref{TC2}) is an isomorphism from $\X \otimes_{\varphi_\Y} \Y$ to $\cj{\tau_\X(\X)\pi_\Y(\Y)}$. \label{TC3}
\item If in addition $\rho_C$ is also faithful then the representation $(\lambda_A, \rho_C, \pi)$ is isometric. \label{TC4}
\end{enumerate}
\end{theorem}
\begin{proof}
Since $\rho_C \colon C \to \Li(\Hi_0)$  is nondegenerate and 
$B$ acts nondegenerately on $\Y$, 
Theorem \ref{CorrespRep} gives a Hilbert space $\Hi_1$, 
a nondegenerate representation $\sigma_B \colon B \to \Li(\Hi_1)$, 
and a 
bounded linear map $\pi_\Y \colon \Y \to \Li(\Hi_0, \Hi_1)$ such that $(\sigma_B, \rho_C, \pi_\Y)$
is a representation of $(\Y, \varphi_\Y)$ on $(\Hi_0, \Hi_1)$. 
Hence, 
a second application of Theorem \ref{CorrespRep} gives a Hilbert 
space $\Hi_2$, a
representation $\lambda_A \colon A \to \Li(\Hi_2)$, 
and a bounded linear map $\tau_\X \colon \X \to \Li(\Hi_1, \Hi_2)$ such that $(\lambda_A, \sigma_B, \tau_\X)$
is a representation of $(\X, \varphi_\X)$ on $(\Hi_1, \Hi_2)$. This takes 
care of part (\ref{TC1}). 
For part (\ref{TC2}), we first prove the existence of the map $\pi$. 
First, fix $n \in \Z_{\geq 1}$, $x_1, \ldots, x_n \in \X$, and 
$y_1, \ldots, y_n \in \Y$. Then, using the fact that $(\sigma_B, \rho_C, \pi_\Y)$ is a representation at the second and third steps  
together with equation (\ref{innerProd})
at the sixth step we find
\begin{align*}
\Bigl\|  \sum_{j=1}^n \tau_\X(x_j)\pi_\Y(y_j) \Bigr\|^2  
& = \sup_{\| \xi \| =1 } \sum_{j,k=1}^n \langle \tau_\X(x_k)\pi_\Y(y_k)\xi, \tau_\X(x_j)\pi_\Y(y_j)\xi \rangle \\
& = \sup_{\| \xi \| =1 } \sum_{j,k=1}^n \langle \pi_\Y(y_j)^*\pi_\Y(\varphi_\Y(\langle x_j, x_k \rangle_B)y_k)\xi, \xi \rangle\\
& = \sup_{\| \xi \| =1 } \sum_{j,k=1}^n \langle \rho_C(\langle y_j , \varphi_\Y(\langle x_j, x_k \rangle_B)y_k \rangle_C)\xi, \xi\rangle\\
& \leq \Bigl\| \sum_{j,k=1}^n  \rho_C(\langle y_j , \varphi_\Y(\langle x_j, x_k \rangle_B)y_k \rangle_C)\Bigr\| \\
& \leq \Bigl\| \sum_{j,k=1}^n  \langle y_j , \varphi_\Y(\langle x_j, x_k \rangle_B)y_k \rangle_C\Bigr\| \\
& = \Bigl\| \Bigl\langle  \sum_{j=1}^n x_j \otimes y_j, \sum_{k=1}^n x_k\otimes y_k\Bigr\rangle_C\Bigr\| \\
& = \Bigl\|  \sum_{j=1}^n x_j \otimes y_j \Bigr\|^2.
\end{align*}
Moreover, a direct computation gives $\tau_\X(xb)\pi_\Y(y)=\tau_\X(x)\pi_\Y(\varphi_\Y(b)y)$
for any $x \in \X, y \in \Y$, and $b \in B$. Therefore, we can extend the map 
$x \otimes y \mapsto \tau_\X(x)\pi_\Y(y)$ to a well defined bounded linear map $\pi \colon \X \otimes_{\varphi_\Y} \Y \to \Li(\Hi_0, \Hi_2)$. 
To finish part (\ref{TC2}), It only remains to check that $(\lambda_A, \rho_C, \pi)$ is indeed a representation
of the correspondence $(\X \otimes_{\varphi_\Y} \Y, \widetilde{\varphi_\X})$
on $(\Hi_0, \Hi_2)$. This follows from three 
 immediate computations on elementary tensors using the fact that 
both  $(\lambda_A, \sigma_B, \tau_\X)$ and  $(\sigma_B, \rho_C, \pi_\Y)$ are representations. 
Part (\ref{TC3}) is now immediate from definition of $\pi$. Finally,
to check part (\ref{TC4}), Remark \ref{FaithIsom} shows that faithfulness of $\rho_C$ is
enough for $(\lambda_A, \rho_C, \pi)$
to be isometric, so we are done. 
\end{proof}

\textbf{Acknowledgments:} The author would like to thank N. Christopher Phillips 
for his advice the past few years and in particular for reading multiple
earlier versions of this paper, always giving useful comments that
significantly improved the exposition. 
He is also grateful to Andrey Blinov for stimulating conversations about the topics 
covered here, some of which led to developing the theory further than originally 
intended. The author also thanks Matthew Daws for pointing out typos 
and in particular for sending useful comments 
that gave rise to Remarks \ref{Daws_Comment3} and \ref{Daws_Comment4}.
Finally, the author thanks the referee their suggestions, 
including one dealing with an updated notation for representations on pairs of Hilbert spaces, 
making it more accessible for future applications.

\end{document}